\documentclass[12pt,leqno]{amsart}
\usepackage{amssymb,amsfonts,amsthm,setspace,indentfirst}

\newtheorem{theo}{Theorem}[section]

\newtheorem{lem}{Lemma}[section]
\newtheorem{pr}{Proposition}[section]

\newcommand{\be}{\begin{equation}}
\newcommand{\ee}{\end{equation}}
\newcommand{\bea}{\begin{eqnarray}}
\newcommand{\eea}{\end{eqnarray}}
\newcommand{\beb}{\begin{eqnarray*}}
\newcommand{\eeb}{\end{eqnarray*}}

\numberwithin{equation}{section}

\begin{document}
\title{Extended weakly symmetric manifolds}
\author{$^{1}$K. K. Baishya, $^{2}$S. K. Jana, $^{3}$M. R. Bakshi, $^{4}$M.
Pain and $^{4}$H. Kundu }
\date{}

\begin{abstract}
In this work, we have introduced and studied some basic geometric properties of extended weakly symmetric
spaces. After classification of this structure we have also established the existence of such a space by presenting a non-trivial example.
\end{abstract}
\maketitle
%%%%%%%%%%%%%%%%%%%%%%%%%%%%%%%%%%%%%%%%%%%%%%%%%%%%%%%%%%%%%%%%%%%%%%%%%%%%%%%%%%%%%%%%%%%%%%%%%%%%%
%%%%%%%%%%%%%%%%%%%%%%%%%%%%%%%%%%%%%%%%%%%%%%%%%%%%%%%%%%%%%%%%%%%%%%%%%%%%%%%%%%%%%%%%%%%%%%%%%%%%%

\footnotetext{%
} 
%%%%%%%%%%%%%%%%%%%%%%%%%%%%%%%%%%%%%%%%%%%%%%%%%%%%%%%%%%%%%%%%%%%%%%%%%%%%%%%%%%%%%%%%%%%%%%%%%%%%%

\section{Introduction}

%%%%%%%%%%%%%%%%%%%%%%%%%%%%
The modern trend of Mathematics is abstraction, generalization, extension
and then applications. Keeping this in our mind we have ideated the notion
of `Extended weakly symmetric manifolds', which extends the notion of
`Generalized weakly symmetric manifolds' \cite{Baishya2} as well as
`Hyper-generalized weakly symmetric manifolds' \cite{Baishya9}, as follows:%
\newline
An $n$-dimensional Riemannian manifold will be termed as \textit{extended
weakly symmetric manifold} (which will be abbreviated hereafter as $(EWS)_{n}
$) if it bears the equation%
\begin{eqnarray}
&&(\nabla
_{X_{5}}R)(X_{1},X_{2},X_{3},X_{4})=A(X_{5})R(X_{1},X_{2},X_{3},X_{4})
\label{ewsng} \\
&&+B(X_{1})R(X_{5},X_{2},X_{3},X_{4})+\bar{B}%
(X_{2})R(X_{1},X_{5},X_{3},X_{4})  \notag \\
&&+D(X_{3})R(X_{1},X_{2},X_{5},X_{4})+\bar{D}%
(X_{4})R(X_{1},X_{2},X_{3},X_{5})  \notag \\
&&+\alpha (X_{5})H(X_{1},X_{2},X_{3},X_{4})+\beta
(X_{1})H(X_{5},X_{2},X_{3},X_{4})  \notag \\
&&+\bar{\beta}(X_{2})H(X_{1},X_{5},X_{3},X_{4})+\gamma
(X_{3})H(X_{1},X_{2},X_{5},X_{4})  \notag \\
&&+\bar{\gamma}(X_{4})H(X_{1},X_{2},X_{3},X_{5})+\theta
(X_{5})G(X_{1},X_{2},X_{3},X_{4})  \notag \\
&&+\phi (X_{1})G(X_{5},X_{2},X_{3},X_{4})+\bar{\phi}%
(X_{2})G(X_{1},X_{5},X_{3},X_{4})  \notag \\
&&+\psi (X_{3})G(X_{1},X_{2},X_{5},X_{4})+\bar{\psi}%
(X_{4})G(X_{1},X_{2},X_{3},X_{5}),  \notag
\end{eqnarray}%
where $R$ is the Riemann curvature tensor, $H=g\wedge S$, $G=g\wedge g$, the
wedge product $q\wedge p$ of two $(0,2)$-tensors is defined by 
\begin{eqnarray*}
(q\wedge p)(X_{1},X_{2},X_{3},X_{4})
&=&q(X_{1},X_{4})p(X_{2},X_{3})-q(X_{1},X_{3})p(X_{2},X_{4}) \\
&&+p(X_{1},X_{4})q(X_{2},X_{3})-p(X_{1},X_{3})q(X_{2},X_{4})
\end{eqnarray*}%
and ($A,B,\bar{B},D,\bar{D},\alpha ,\beta ,\bar{\beta},\gamma ,\bar{\gamma}%
,\theta ,\phi ,\bar{\phi},\psi ,\bar{\psi})$ are non-zero 1-forms.\newline
The beauty of such space is that it has the taste of

(i) symmetric space \cite{Cartan} for ($0,0,0,0,0,0,0,0,0,0,0,0,0,0,0)$,

(ii) recurrent space \cite{Waiker} for ($A,0,0,0,0,0,0,0,0,0,0,0,0,0,0)$,

(iii) hyper generalized recurrent space \cite{ShaikhP} for ($A,0,0,0,0,\alpha
,0,0,0,0,0,0,0,0,0)$,

(iv) generalized recurrent space \cite{Dubey} for ($A,0,0,0,0,0,0,0,0,0,%
\theta ,0,0,0,0)$,

(v) pseudo symmetric space \cite{Chaki1} for $\frac{A}{2}=B=D=\delta \neq 0 $%
, $\alpha =\beta =\gamma =0$,

(vi) generalized pseudo symmetric space \cite{Baishya1} for $\frac{A}{2}%
=B=D=\delta \neq 0$, $\frac{\alpha }{2}=\beta =\gamma =\mu $,

(vii) semi-pseudo symmetric space \cite{Tarafdar} for $B=D=\delta $, $%
A=\alpha =\beta =\gamma =0$,

(viii) generalized semi-pseudo symmetric space for $A=0=\alpha ,B=D=\delta $%
, $\beta =\gamma =\mu $,

(ix) almost pseudo symmetric space \cite{Chaki} for $A=E+H,B=D=H$, $\alpha
=\beta =\gamma =0$,

(x) almost generalized pseudo symmetric space (\cite{Baishya6}, \cite%
{Baishya7}, \cite{Baishya8}) for $A=E+H,B=D=H$, $\alpha =\lambda +\psi
,\beta =\gamma =\lambda $,

(xi) weakly symmetric space \cite{Tam} for $A,B,D\neq 0$, $\alpha =\beta
=\gamma =0$.

(xii) generalized weakly symmetric space (\cite{Baishya2}, \cite{Baishya3}, 
\cite{Baishya4}, \cite{Baishya5}) for $A,B,D\neq 0$, $\alpha =\beta
=\gamma =0$.

(xiii) hyper generalized semi-pseudo symmetric space (for $\bar{A}=\bar{%
\alpha}=0,\ \bar{B}=\bar{D}\neq 0$, $\bar{\beta}=\bar{\gamma}\neq 0$),

(xiv) hyper generalized pseudo symmetric space \cite{Baishya11} for $\bar{A%
}=2\bar{B}=2\bar{D},$ $\bar{\alpha}=2\bar{\beta}=2\bar{\gamma}$,

(xvi) almost hyper generalized pseudo symmetric space for $\bar{A}=\bar{B}%
+H_{1},H_{1}=\bar{B}=\bar{D}\neq 0$, $\bar{\alpha}=\bar{\beta}+H_{2},H_{2}=%
\bar{\beta}=\bar{\gamma}\neq 0$ and

(xvii) hyper generalized weakly symmetric space (\cite{Baishya9}, \cite%
{Baishya10}) for $\bar{A}=\bar{B}+H_{1},\ H_{1}=\bar{B}=\bar{D}\neq 0$, $%
\bar{\alpha}=\bar{\beta}+H_{2},\ H_{2}=\bar{\beta}=\bar{\gamma}\neq 0$.

Firstly we have put some basic results in the Section 2. Section 3 is
devoted to the discussion of general properties of a $(EWS)_{n}$. The last section is concerned with a
non-trivial example of $(EWS)_{n}$. 

%============================

\section{Some basic results}

%===========================

\begin{pr}
\label{pr3.1} Let $K$ be a generalized curvature tensor and $A$%
, $B$ be two 1-forms. Then 
\begin{eqnarray*}
&&A(X_5) K(X_1,X_2,X_3,X_4) + B(X_5) K(X_1,X_4,X_2,X_3) \\
&&= \left\{A(X_5)-\frac{1}{2}B(X_5)\right\} K(X_1,X_2,X_3,X_4).
\end{eqnarray*}
\end{pr}

%=================================================================

\begin{lem}
\label{lem3.1} Let $K$ be a generalized curvature tensor on a
semi-Riemannian manifold $M$. If 
\begin{equation}  \label{eq1}
A(X_1) K(X_2,X_5,X_3,X_4) + A(X_2) K(X_1,X_5,X_3,X_4) = 0
\end{equation}
then either $A=0$ or $K=0$.
\end{lem}

\noindent Proof: Let $A\ne 0$. So there exists a vector field $\xi$ such
that $A(\xi)\ne 0$. Then putting $X_1=X_2=\xi$ in \eqref{eq1}, we get 
\begin{equation*}
A(\xi) K(\xi,X_5,X_3,X_4) = 0 \ \ \Rightarrow \ \ K(\xi,X_5,X_3,X_4) = 0.
\end{equation*}
Now putting $X_1=\xi$ in \eqref{eq1}, we get 
\begin{equation*}
A(\xi) K(X_2,X_5,X_3,X_4)+A(X_2) K(\xi,X_5,X_3,X_4)=0,
\end{equation*}
which implies $K(X_2,X_5,X_3,X_4)=0$. This completes the proof. 
%==============================================================

\begin{pr}
\label{pr3.2} If on a semi-Riemannian manifold $M$ 
\begin{eqnarray}  \label{condG}
&&A(X_1) R(X_2,X_5,X_3,X_4) + A(X_2) R(X_1,X_5,X_3,X_4) \\
&&+ B(X_1) G(X_2,X_5,X_3,X_4) + B(X_2) G(X_1,X_5,X_3,X_4) = 0  \notag
\end{eqnarray}
holds, then either $M$ is of constant curvature or $A=B=0$.
\end{pr}

\noindent Proof: Contracting \eqref{condG} over $X_1$ and $X_2$, we get 
\begin{equation*}
R(X_A,X_5,X_3,X_4)+G(X_B,X_5,X_3,X_4)=0,
\end{equation*}
where $X_A, X_B$ denote the associated vector fields of $A$ and $B$
respectively. Again contraction over $X_{1}$ and $X_{3}$ on \eqref{condG},
we obtain 
\begin{eqnarray*}
&&\ \ B(X_2) \{g(X_4,X_5) - n g(X_4,X_5)\} - A(X_2) S(X_4,X_5) \\
&&\hspace{0.8cm} + R(X_A,X_4,X_2,X_5) + G(X_B,X_4,X_2,X_5) = 0 \\
&&\Rightarrow A(X_2) S(X_4,X_5) + (n-1) B(X_2) g(X_4,X_5) = 0 \\
&&\Rightarrow r A(X_2) + n(n-1) B(X_2) = 0.
\end{eqnarray*}
Now, using this value in \eqref{condG}, we get 
\begin{eqnarray*}
&&A(X_{1})\left[R (X_{2},X_{5},X_{3},X_{4})-\frac{r}{n(n-1)}G
(X_{2},X_{5},X_{3},X_{4})\right] \\
&&+A(X_{2})\left[R(X_{1},X_{5},X_{3},X_{4})-\frac{r}{n(n-1)}%
G(X_{1},X_{5},X_{3},X_{4})\right] =0.
\end{eqnarray*}%
Now by virtue of the Lemma \ref{lem3.1}, one can easily bring out either $%
A=0 $ or $M$ is of constant curvature. Moreover for the case of $A=0$, we
have $B=-\frac{r}{n(n-1)}A=0$. 
%=====================================================================================

\begin{pr}
\label{pr3.3} Let us suppose that on a semi-Riemannian manifold $M$ 
\begin{eqnarray}  \label{condH}
&&A(X_1) R(X_2,X_5,X_3,X_4) + A(X_2) R(X_1,X_5,X_3,X_4) \\
&&+ B(X_1) H(X_2,X_5,X_3,X_4) + B(X_2) H(X_1,X_5,X_3,X_4)  \notag \\
&&+ D(X_1) G(X_2,X_5,X_3,X_4) + D(X_2) G(X_1,X_5,X_3,X_4) = 0  \notag
\end{eqnarray}
holds, where $H=g\wedge S$. Then either $M$ is of constant curvature or
conformally flat or $A=B=D=0$ or $A=\frac{r}{n} B+D=0$.
\end{pr}

\noindent Proof: Contracting \eqref{condH} over $X_1$ and $X_2$, we get 
\begin{equation*}
R(X_A,X_5,X_3,X_4)+H(X_B,X_5,X_3,X_4)+G(X_D,X_5,X_3,X_4)=0,
\end{equation*}
where $X_A, X_B, X_D$ denote the associated vector fields of $A$, $B$ and $D$
respectively. Again contracting \eqref{condH} over $X_{1}$ and $X_{3}$,
we obtain 
\begin{eqnarray*}
&& D(X_2) (1-n)g(X_4,X_5) + B(X_2) \{(2-n) S(X_4,X_5)-r g(X_4,X_5)\}-A(X_2) S(X_4,X_5) \\
&&+ R(X_A,X_4,X_2,X_5)+ H(X_B,X_4,X_2,X_5)+G(X_D,X_4,X_2,X_5) = 0,
\end{eqnarray*}
which gives 
\begin{eqnarray}
 \ \{A(X_2)+(n-2)D(X_2)\} S(X_4,X_5) = \{(1-n)B(X_2)-r D(X_2)\} g(X_4,X_5). \label{eqsg}
\end{eqnarray}
Now two cases arise\newline
Case 1: Let $A+(n-2)D\ne 0$. Then the manifold becomes Einstein and hence %
\eqref{condH} becomes 
\begin{eqnarray*}
&&A(X_1) R(X_2,X_5,X_3,X_4)+ \left\{\frac{r}{n} B(X_1)+D(X_1)\right\} G(X_2,X_5,X_3,X_4)
\\
&& + A(X_2) R(X_1,X_5,X_3,X_4) + \left\{\frac{r}{n} B(X_2)+D(X_2)\right\}
G(X_1,X_5,X_3,X_4) = 0.  \notag
\end{eqnarray*}
Thus using Proposition \ref{pr3.3}, we can say that either $A=\frac{r}{n}
B+D=0$ or $M$ is of constant curvature.\newline
Case 2: Let $A+(n-2)D = 0$. Then from \eqref{eqsg}, we have $n B+D =0$. Thus
putting the expression of $A$ and $B$ in terms of $D$, we can write the
equation \eqref{condH} as 
\begin{eqnarray}
D(X_1) C(X_2,X_5,X_3,X_4) + D(X_2) C(X_1,X_5,X_3,X_4) = 0.
\end{eqnarray}
Thus using Lemma \ref{lem3.1}, we can say that either $M$ is of conformally
flat or $D=0$. Moreover then $A=(2-n)D=0$ and $B=-\frac{1}{n}D=0$. %
%%%%%%%%%%%%%%%%%%%%%%%%%%%%%%%%%%%%%%%%%%%%%%%%%%%%%%%%%%%%%%%%%%%%%%%%%%%%%%%%%%%%%%%%%%%%%%%%%%%%%%

\section{$(EWS)_{n}$}

%%%%%%%%%%%%%%%%%%%%%%%%
Using the symmetry properties of $R$, viz., $%
R(X_1,X_2,X_3,X_4)=-R(X_2,X_1,X_3,X_4)$ on \eqref{ewsng} we get 
\begin{eqnarray*}
&&\{B(X_1)-\bar B(X_1)\} R(X_2,X_5,X_3,X_4) + \{B(X_2)-\bar B(X_2)\}
R(X_1,X_5,X_3,X_4) \\
&&+ \{\beta(X_1)-\bar \beta(X_1)\} H(X_2,X_5,X_3,X_4) + \{\beta(X_2)-\bar
\beta(X_2)\} H(X_1,X_5,X_3,X_4) \\
&&+ \{\phi(X_1)-\bar \phi(X_1)\} G(X_2,X_5,X_3,X_4) + \{\phi(X_2)-\bar
\phi(X_2)\} G(X_1,X_5,X_3,X_4) = 0.
\end{eqnarray*}
Then using proposition \ref{pr3.3}, we can say that either the space is
conformally flat or constant curvature or generalized weakly symmetric or we
get the following relations between the associated 1-forms: 
\begin{equation*}
B=\bar B, \ D=\bar D, \ \beta =\bar \beta, \ \gamma=\bar\gamma \mbox{ and }
\phi = \bar\phi,
\end{equation*}
which leads us to state the following:

\begin{theo}
Suppose $M$ is a non-conformally flat proper extended weakly symmetric
manifold (i.e., not generalized weakly symmetric). The defining equation of
a $(EWS)_n$ is reduced as the following: 
\begin{eqnarray}  \label{rewsn}
&& (\nabla _{X_5}R)(X_1,X_2,X_3,X_4)=A(X_5)R(X_1,X_2,X_3,X_4) \\
&&+B(X_1)R(X_5,X_2,X_3,X_4)+B(X_2)R(X_1,X_5,X_3,X_4)  \notag \\
&&+D(X_3)R(X_1,X_2,X_5,X_4)+D(X_4)R(X_1,X_2,X_3,X_5)  \notag \\
&&+\alpha(X_5)H(X_1,X_2,X_3,X_4) +\beta(X_1)H(X_5,X_2,X_3,X_4)  \notag \\
&&+\beta(X_2)H(X_1,X_5,X_3,X_4)+\gamma(X_3)H(X_1,X_2,X_5,X_4)  \notag \\
&&+\gamma(X_4)H(X_1,X_2,X_3,X_5)+\theta(X_5)G(X_1,X_2,X_3,X_4)  \notag \\
&&+\phi(X_1)G(X_5,X_2,X_3,X_4)+\phi(X_2)G(X_1,X_5,X_3,X_4)  \notag \\
&&+\psi(X_3)G(X_1,X_2,X_5,X_4)+\psi(X_4)G(X_1,X_2,X_3,X_5).  \notag
\end{eqnarray}
\end{theo}

Now contracting (\ref{rewsn}) we get 
\begin{eqnarray}
&&(\nabla _{X_{5}}S)(X_{2},\ X_{3})  \label{4a} \\
&=&f_{1}(X_{5})S(X_{2},\ X_{3})+f_{2}(X_{2})S(X_{5},\
X_{3})+f_{3}(X_{3})S(X_{2},X_{5})  \notag \\
&&+f_{4}(X_{5})g(X_{2},\ X_{3})+f_{_{5}}(X_{2})g(X_{5},\
X_{3})+f_{6}(X_{3})g(X_{2},\ X_{5})  \notag \\
&&-B(R(X_{2},X_{5})X_{3})+D(R(X_{5},\ X_{3})X_{2})+[\beta (LX_{5})+\gamma
(LX_{5})]g(X_{2},\ X_{3})  \notag \\
&&-\beta (LX_{2})g(X_{5},\ X_{3})-\gamma (LX_{3})g(X_{5},\ X_{2}),  \notag
\end{eqnarray}%
where $L$ is the Ricci operator and
\begin{eqnarray}
&&f_{1}=A+(n-2)\alpha +\beta +\gamma ,\ \ f_{2}=B+(n-3)\beta ,  \label{5a} \\
&&f_{3}=D+(n-3)\gamma ,\ \ f_{4}=(n-1)\theta +\phi +\psi +r\alpha ,  \notag
\\
&&f_{5}=(n-2)\phi +r\beta ,\ \ f_{6}=(n-2)\psi +r\gamma .  \notag
\end{eqnarray}%
Thus we can conclude that

\begin{theo}
Every ($EWS$)$_{n}$ is ($GWRS$)$_{n}$ if %
\begin{equation}
B(LX)+D(LX)+(n-1)[\beta(LX)+\gamma(LX)]=0.  \label{6a}
\end{equation}
\end{theo}

\begin{proof}
From (\ref{4a}) it is obvious that ($EWS$)$_{n}$ becomes ($GWRS$)$_{n}$ if 
\begin{eqnarray*}
&&D(R(X,\ W)Y)+[\beta(LX)+\gamma(LX)]g(Y,\ W) \\
&=&\beta(LY)g(X,\ W)+\gamma(LW)g(X,\ Y)+B(R(Y,X)W),
\end{eqnarray*}%
which, on contraction, gives (\ref{6a}).
\end{proof}

%%%%%%%%%%%%%%%%%%%%%%%%%%%%%%%%%%%%%%%%%%%%%%%%%%%%%%%%%%%%%%%%%%%%%%%%%%%%%%%%%%%%%%%%%%%%%%%%%%%%%%
%\section{Form of $R$}
%%%%%%%%%%%%%%%%%%%%%%
Let us now consider the Bianchi's second identity 
\begin{equation*}
(\nabla _{X_{5}}R)(X_{1},X_{2},X_{3},X_{4})+(\nabla
_{X_{1}}R)(X_{2},X_{5},X_{3},X_{4})+(\nabla
_{X_{2}}R)(X_{5},X_{1},X_{3},X_{4})=0.
\end{equation*}%
Then using \eqref{rewsn} on the above we get 
\begin{eqnarray}
&&\omega
(X_{5})G(X_{1},X_{2},X_{3},X_{4})+\omega(X_{2})G(X_{5},X_{1},X_{3},X_{4})
\label{rfb2} \\
&&+\omega (X_{1})G(X_{2},X_{5},X_{3},X_{4})+\delta
(X_{5})H(X_{1},X_{2},X_{3},X_{4})  \notag \\
&&+\delta (X_{2})H(X_{5},X_{1},X_{3},X_{4})+\delta
(X_{1})H(X_{2},X_{5},X_{3},X_{4})  \notag \\
&&+J(X_{5})R(X_{1},X_{2},X_{3},X_{4})+J(X_{2})R(X_{5},X_{1},X_{3},X_{4}) 
\notag \\
&&+J(X_{1})R(X_{2},X_{5},X_{3},X_{4})=0,  \notag
\end{eqnarray}%
where $J=A-2B,\ \delta =\alpha -2\beta $ and $\omega =\theta -2\phi $.%
\newline
Now taking contraction over $X_{2}$ and $X_{3}$ on \eqref{rfb2}, we get 
\begin{eqnarray}  \label{Rf}
&& R(X_{J},X_{4},X_{1},X_{5}) \\
&&=\{(n-2)\omega (X_{5})+r\delta (X_{5})-S(X_{5},X_{\delta })\}g(X_{1},X_{4})
\notag \\
&&-\{(n-2)\omega (X_{1})+r\delta (X_{1})-S(X_{1},X_{\delta })\}g(X_{4},X_{5})
\notag \\
&&+\{J(X_{5})+(n-3)\delta (X_{5})\}S(X_{1},X_{4})  \notag \\
&&-\{J(X_{1})+(n-3)\delta (X_{1})\}S(X_{4},X_{5}).  \notag
\end{eqnarray}%
Further replacing $X_{5}$ by $X_{J}$ in \eqref{rfb2} and using the form of $%
R(X_{J},X_{4},X_{2},X_{5})$ from above, we get 
\begin{eqnarray*}
&&J(X_{J})R(X_{1},X_{2},X_{3},X_{4}) \\
&&= 2\omega
(X_{J})\{g(X_{2},X_{3})g(X_{1},X_{4})-g(X_{1},X_{3})g(X_{2},X_{4})\} \\
&&+2\delta (X_{J})\{g(X_{1},X_{4})S(X_{2},X_{3})+g(X_{2},X_{3})S(X_{1},X_{4})
\\
&&-g(X_{1},X_{3})S(X_{2},X_{4})-g(X_{2},X_{4})S(X_{1},X_{3})\} \\
&&+g(X_{1},X_{3})\{S(X_{J},X_{4})\delta (X_{2})+S(J,X_{2})\delta
(X_{4})+J(X_{4})\epsilon (X_{2})+J(X_{2})\epsilon (X_{4})\} \\
&&+g(X_{1},X_{4})\{-S(X_{J},X_{3})\delta (X_{2})-S(J,X_{2})\delta
(X_{3})+J(X_{3})\epsilon (X_{2})+J(X_{2})\epsilon (X_{3})\} \\
&&+g(X_{2},X_{3})\{-S(X_{J},X_{4})\delta (X_{1})-S(J,X_{1})\delta
(X_{4})+J(X_{4})\epsilon (X_{1})+J(X_{1})\epsilon (X_{4})\} \\
&&+g(X_{2},X_{4})\{S(X_{J},X_{3})\delta (X_{1})+S(J,X_{1})\delta
(X_{3})+J(X_{3})\epsilon (X_{1})+J(X_{1})\epsilon (X_{3})\} \\
&&+S(X_{1},X_{3})\{2J(X_{2})J(X_{4})+(n-2)J(X_{4})\delta
(X_{2})+(n-2)J(X_{2})\delta (X_{4})\} \\
&&+S(X_{1},X_{4})\{-2J(X_{2})J(X_{3})-(n-2)J(X_{3})\delta
(X_{2})-(n-2)J(X_{2})\delta (X_{3})\} \\
&&+S(X_{2},X_{3})\{-2J(X_{1})J(X_{4})-(n-2)J(X_{4})\delta
(X_{1})-(n-2)J(X_{1})\delta (X_{4})\} \\
&&+S(X_{2},X_{4})\{2J(X_{1})J(X_{3})+(n-2)J(X_{3})\delta
(X_{1})+(n-2)J(X_{1})\delta (X_{3})\},
\end{eqnarray*}%
where $\epsilon (X)=S(\delta,X)-(n-1)\omega (X)-r\delta (X)$.\\
Thus if $J$ is non-null, then we can write the form of the curvature
tensor $R$ as follows: 
\begin{equation}  \label{curvature}
R=\frac{1}{2J(X_{J})}\left[ 2\delta (X_{J})g\wedge S+2\omega (X_{J})g\wedge
g+E\wedge g+F\wedge S\right],
\end{equation}%
where 
\begin{equation*}
E(X_{1},X_{2})=S(X_{J},X_{1})\delta (X_{2})+S(X_J,X_{2})\delta
(X_{1})+J(X_{1})\epsilon (X_{2})+J(X_{2})\epsilon (X_{1}),
\end{equation*}%
\begin{equation*}
F(X_{1},X_{2})=2J(X_{2})J(X_{1})+(n-2)J(X_{1})\delta
(X_{2})+(n-2)J(X_{2})\delta (X_{1}).
\end{equation*}%
%
%%%%%%%%%%%%%%%%%%%%%%%%%%%%%%%%%%%%%%%%%%%%%%%%%%%%%%%%%%%%%%%%%%%%%%%%%%%%%%%%%%%%%%%%%%%%%%%%%%%
%%%%%%%%%%%%%%%%%%%%%%%%%%%%%%%%%%%%%%%%%%%%%%%%%%%%%%%%%%%%%%%%%%%%%%%%%%%%%%%%%%%%%%%%%%%%%%%%%%%%
Next, contracting the above we have%
\begin{eqnarray}
&&rJ(X_{2})J(X_{3})-\{J(X_{2})J(S(X_{J},X_{3}))+J(X_{3})J(S(X_{J},X_{2}))\}
\label{cnR} \\
&&-\frac{(n-1)(n-2)}{2}\{J(X_{3})w(X_{2})+J(X_{2})w(X_{3})\}  \notag \\
&&+(n-2)r\{J(X_{3})\delta (X_{2})+J(X_{2})\delta (X_{3})\}  \notag \\
&&-(n-2)\{J(X_{3})\delta (S(X_J,X_{2}))+J(X_{2})\delta (S(X_J,X_{3}))\}=0. 
\notag
\end{eqnarray}%
Again contracting \eqref{Rf}, we have 
\begin{eqnarray}  \label{scalar}
&&r[J(X_{3})+2(n-2)\delta (X_{3})]-2J(S(X_{J},X_{3}))+(n-1)(n-2)w(X_{3}) \\\nonumber
&&-2(n-2)\delta (S(X_J,X_{3}))=0.
\end{eqnarray}%
Thus we state the following:

\begin{theo}
In a $(EWS)_{n}$ the Riemann curvature tensor and scalar curvature are given
by (\emph{\ref{curvature}}) and (\emph{\ref{scalar}}) respectively, provided 
$J$ is non-null.
\end{theo}

%%%%%%%%%%%%%%%%%%%%%%%%%%%%%%%%%%%%%%%%%%%%%%%%%%%%%%%%%%%%%%%%%%%%%%%%%%%%%%%%%%%%%%%%%%%%%%%%%%%%
%\section{R-harmonic (EWS)$_{n}$}
%%%%%%%%%%%%%%%%%%%%%%%%%%%%%%%%%%%%%%%%%%%%%%%%%%%%%%%%%%%%%%%%%%%%%%%%%%%%%%%%%%%%%%%%%%%%%%%%%%%%%
%%%%%%%%%%%%%%%%%%%%%%%%%%%%%%%%%%%%%%%%%%%%%%%%%%%%%%%%%%%%%%%%%%%%%%%%%%%%%%%%%%%%%%%%%%%%%%%%%%%%%

\section{Example}

%=================
\textbf{Example}: Let $(\mathbb{R}^{4},g)$ be a $4$-dimensional space endowed
with the metric $g$ given by 
\begin{equation}  \label{met1}
ds^{2}=(dx_{1})^{2}+x_1(dx_{2})^{2}+x_4 (dx_{3})^{2}+x_3(dx_{4})^{2}.
\end{equation}%
The non-zero components of Riemann curvature $R$, Ricci tensor $S$, Gaussian
curvature tensor $G$, and $g\wedge S$ are 
\begin{equation*}
R_{1212}=\frac{1}{4 x_1}, \ \ R_{3434}=\frac{x_3+x_4}{4 x_3 x_4};
\end{equation*}
\begin{equation*}
S_{11}=-\frac{1}{4 x_1^2}, \ S_{22}=-\frac{1}{4 x_1}, \ S_{33}=-\frac{x_3+x_4%
}{4 x_3^2 x_4}, \ S_{44}=-\frac{x_3+x_4}{4 x_3 x_4^2};
\end{equation*}
\begin{equation*}
G_{1212}=-2 x_1, \ G_{1313}=-2 x_4, \ G_{1414}=-2 x_3,
\end{equation*}
\begin{equation*}
G_{2323}=-2 x_1 x_4, \ G_{2424}=-2 x_1 x_3, \ G_{3434}=-2 x_3 x_4;
\end{equation*}
\begin{equation*}
(g\wedge S)_{1212}=\frac{1}{2 x_1}, \ \ (g\wedge S)_{1313}=\frac{x_3
x_1^2+x_4 x_1^2+x_3^2 x_4^2}{4 x_1^2 x_3^2 x_4},
\end{equation*}
\begin{equation*}
(g\wedge S)_{1414}=\frac{x_3 x_1^2+x_4 x_1^2+x_3^2 x_4^2}{4 x_1^2 x_3 x_4^2}%
, \ \ (g\wedge S)_{2323}=\frac{x_3 x_1^2+x_4 x_1^2+x_3^2 x_4^2}{4 x_1 x_3^2
x_4},
\end{equation*}
\begin{equation*}
(g\wedge S)_{2424}=\frac{x_3 x_1^2+x_4 x_1^2+x_3^2 x_4^2}{4 x_1 x_3 x_4^2},
\ \ (g\wedge S)_{3434}=\frac{x_3+x_4}{2 x_3 x_4}.
\end{equation*}
Covariant derivatives of Riemann curvature tensor and Ricci tensor are 
\begin{equation*}
R_{1212,1}=-\frac{1}{2 x_1^2}, \ R_{3434,3}=\frac{x_3+2 x_4}{4 x_3^2 x_4}, \
R_{3434,4}=\frac{2 x_3+x_4}{4 x_3 x_4^2};
\end{equation*}
\begin{equation*}
S_{11,1}=\frac{1}{2 x_1^3}, \ S_{22,1}=\frac{1}{2 x_1^2}, \ S_{33,3}=\frac{%
x_3+2 x_4}{4 x_3^3 x_4},
\end{equation*}
\begin{equation*}
S_{33,4}=\frac{2 x_3+x_4}{4 x_3^2 x_4^2}, \ S_{44,3}=\frac{x_3+2 x_4}{4
x_3^2 x_4^2}, \ S_{44,4}=\frac{2 x_3+x_4}{4 x_3 x_4^3}.
\end{equation*}
Then it can be easily check that the space satisfies \eqref{rewsn} for the
following choice of the 1-forms 
\begin{equation*}
A=\left\{\frac{2 x_3^2 x_4^2}{x_{1}\lambda_{1}},0,-\frac{x_1^2 \left(x_3+2
x_4\right)}{x_3 \lambda_{1}},-\frac{x_1^2 \left(2 x_3+x_4\right)}{x_4
\lambda_{1}}\right\},
\end{equation*}
\begin{equation*}
B=\left\{\frac{2 x_1 x_3^2 x_4^2 \left(x_3+x_4\right)}{\lambda_{2}},0,-\frac{%
x_1^2 x_3 x_4^2 \left(x_3+2 x_4\right)}{\lambda_{2}},-\frac{x_1^2 x_3^2 x_4
\left(2 x_3+x_4\right)}{\lambda_{2}}\right\},
\end{equation*}
\begin{equation*}
D=\left\{-\frac{x_3+x_4}{4 x_{1}\lambda_{3}},0,\frac{x_3+2 x_4}{8 x_3
\lambda_{3}},\frac{2 x_3+x_4}{8 x_4 \lambda_{3}}\right\},
\end{equation*}
\begin{equation*}
\alpha=\left\{\frac{\lambda_{3}}{x_1^2 \left(x_3+x_4\right)}%
,0,0,0\right\}=-\theta,
\end{equation*}
\begin{equation*}
\beta=\left\{1,0,0,0\right\}=-\phi, \ \ \gamma=\left\{\frac{\lambda_{1}}{%
8x_{1}^{2}x_3^2 x_4^2},0,0,0\right\}=-\psi,
\end{equation*}
where $\lambda_{1}=x_3^2 x_4^2+x_1^2\left(x_3+x_4\right)$, $%
\lambda_{2}=x_1^4 \left(x_3+x_4\right)^2-x_3^4 x_4^4$ and $\lambda_{3}=x_3^2
x_4^2-x_1^2 \left(x_3+x_4\right)$. 
%%%%%%%%%%%%%%%%%%%%%%%%%%%%%%%%%%%%%%%%%%%%%%%%%%%%%%%%%%%%%%%%%%%%%%%%%%%%%%%%%%%%%%%%%%%%%%%%%%%%
%%%%%%%%%%%%%%%%%%%%%%%%%%%%%%%%%%%%%%%%%%%%%%%%%%%%%%%%%%%%%%%%%%%%%%%%%%%%%%%%%%%%%%%%%%%%%%%%%%%%

\end{document}